\documentclass[11pt]{article}
\usepackage{amsmath,amssymb,amsthm}
\usepackage{graphicx,subfigure,float,url}
\usepackage{pdfsync}
\usepackage[english]{babel}
\usepackage[utf8]{inputenc}
\usepackage{enumitem}
\usepackage{fancybox}
\usepackage{listings}
\usepackage{mathtools}
\usepackage{amsfonts}
\usepackage[bottom]{footmisc}
\usepackage{verbatim}

\usepackage{float}
\usepackage{makeidx}

\normalbaselines

\newtheorem{Theorem}{Theorem}[section]

\newtheorem{Proposition}[Theorem]{Proposition}

\newtheorem{Lemma}[Theorem]{Lemma}

\newtheorem{Corollary}[Theorem]{Corollary}

\newtheorem{Remark}[Theorem]{Remark}

\newcommand{\RR}{{{\rm I} \kern -.15em {\rm R} }}

\newcommand{\C}{{{\rm l} \kern -.42em {\rm C} }}

\newcommand{\nat}{{{\rm I} \kern -.15em {\rm N} }}

\newcommand{\be}{\begin{equation}}
\newcommand{\ee}{\end{equation}}
\newcommand{\beq}{\begin{eqnarray}}
\newcommand{\eeq}{\end{eqnarray}}
\newcommand{\beqs}{\begin{eqnarray*}}
\newcommand{\eeqs}{\end{eqnarray*}}
\newcommand{\bt}{\begin{Theorem}}
\newcommand{\et}{\end{Theorem}}
\newcommand{\br}{\begin{Remark}}
\newcommand{\er}{\end{Remark}}
\newcommand{\bc}{\begin{Corollary}}
\newcommand{\ec}{\end{Corollary}}
\newcommand{\bl}{\begin{Lemma}}
\newcommand{\el}{\end{Lemma}}
\newcommand{\bd}{\begin{definition}}
\newcommand{\ed}{\end{definition}}

\renewcommand{\geq}{\geqslant}
\renewcommand{\ge}{\geqslant}
\renewcommand{\leq}{\leqslant}
\renewcommand{\le}{\leqslant}

\newcommand{\W}{\mathcal{H}}

\title{Exponential decay estimates for semilinear wave-type \\equations with time-dependent time delay}
\author{
Cristina Pignotti\footnote{Dipartimento di Ingegneria e Scienze dell'Informazione e Matematica, Universit\`{a} di L'Aquila, Via Vetoio, Loc. Coppito, 67100 L'Aquila Italy (\texttt{cristina.pignotti@univaq.it}).}
}

\date{}

\begin{document}

\textwidth=160 mm

\textheight=225mm

\parindent=8mm

\frenchspacing

\maketitle

\begin{abstract}
In this paper we  analyze a semilinear  damped second order evolution equation with time-dependent time delay and time-dependent delay feedback coefficient. The nonlinear term satisfies a local Lipschitz continuity assumption.  Under appropriate conditions, we prove well-posedness and exponential stability of our model for \emph{small} initial data. Our arguments combine a  Lyapunov functional approach with some continuity arguments. Moreover, as an application of our abstract results, the damped wave equation with a source term and delay feedback is analyzed.
\end{abstract}

\vspace{5 mm}

\def\qed{\hbox{\hskip 6pt\vrule width6pt
height7pt
depth1pt  \hskip1pt}\bigskip}



\section{Introduction}
In recent years the study of evolution equations with time delay effects attracted the interest of several researchers. Often, in mathematical models, a time delay has to be included in order to take into account some lags present in real phenomena such as reaction times,  gestation times, times for maturation, etc.
Concerning wave-type equations, it is well-known that time delays can generate instability phenomena (see e.g.  \cite{Datko, NPSicon06}). 
 However, it is also now well-known that appropriate feedback laws can restitute the stability properties of the undelayed model (see \cite{NPSicon06, XYL}).
Here, we are interested in proving exponential stability estimates for second-order semilinear damped evolution equations in presence of a time delay feedback. In particular, we are interested in time-variable time delays, namely the time delay is not a fixed constant but it depends upon the time variable.   

Let $H$ be a Hilbert space and let $A:{\mathcal D}(A)\subset H \to H$ be a positive self-adjoint operator with dense domain and compact inverse in $H.$ Let us consider the following wave-type equation:
\begin{equation}\label{modello}
\begin{array}{l}
\displaystyle{u_{tt}(t)+A u(t)+CC^*u_t(t)+k(t)BB^* u_t(t-\tau(t))=\nabla \psi(u(t)),\quad t\geq 0,}\\
\displaystyle{u(0)=u_0, \quad u_t(0)=u_1,}\\ 
\displaystyle{B^*u_t(s)=g(s), \quad s\in [-\overline\tau,0], }
\end{array}
\end{equation}
where $\tau: [0, +\infty)\rightarrow [0,+\infty)$ is the time delay function, i.e. a continuous function which  satisfies
\begin{equation}\label{taubar}
\tau(t)\le\overline\tau.
\end{equation}
 The damping coefficients $k(\cdot)$ is assumed to  belong to  $L^1_{loc}([-\overline\tau,+\infty)),$ and $(u_0, u_1, g)$ are the initial data taken in suitable spaces. Moreover, for given real Hilbert spaces $W_1$ and $W_2$ that  will be identified with their dual spaces,  $C:W_1\to H$ and $B:W_2\to H$ are bounded linear operators.   We will denote    $C^*$ and  $B^*$ the adjoint of $C$ and $B$ respectively.  
We assume that the damping operator $CC^*$ satisfies a control geometric property (see e.g. \cite{Bardos} or \cite[Chapter 5]{K}). Moreover, on the delay feedback coefficient, we assume that the integrals on intervals of length $\overline\tau$ are uniformly bounded, namely, 
\begin{equation}\label{damp_coeff}
\int_{t-\overline{\tau}}^t |k(s)|ds\leq K, \quad \forall \ t\ge \overline 0,
\end{equation}
for some $K>0.$

We are interested in studying the well-posedness of the above system and in proving an exponential stability result for solutions corresponding to enough small initial data.

Problem \eqref{modello} has been recently studied in \cite{PP_DCDSS} in the case of a constant time delay. We extend here such an analysis to the case of a time variable time delay. The extension is nontrivial since several difficulties have to be overcome in order to deal with a time dependent function $\tau.$ First, several arguments based on the so-called step by   
step procedure, typical to approach time delay models, cannot be reproduced; second, the time delay can also degenerate and, when this happens, namely when $\tau(t)=0,$ our model is not more delayed. We have to take into account these different features in our arguments.
However, we are able to extend the results of \cite{PP_DCDSS} to the new setting.
Moreover, our arguments, allow us to improve also the results in the case of a constant time delay, giving a proof simpler. 

A linear version of such a model, in case of a constant time delay,  has been first studied in \cite{SCL12},  where a wave equation with frictional damping and delay feedback with a constant coefficient has been analyzed. Under a suitable smallness condition on the delay term coefficient, an exponential decay estimate has been proven. This result has then  been extended to linear  wave equations with boundary dissipative condition (see \cite{ANP10}) and with viscoelastic damping (see \cite{AlNP, guesmia}). We quote also  \cite{JEE15, NicaisePignotti18, KP} for related stability results for abstract semilinear evolution equations. However,   in the nonlinear setting, the results previously obtained require that the damping operator $CC^*$ contrasts, in the spirit of \cite{NPSicon06} (cf. also \cite{XYL}), the delay feedback. Indeed, in \cite{JEE15, NicaisePignotti18}, where the delay coefficient is constant, i.e. $k(t)\equiv k,$  in order to have a not increasing energy, it is assumed $\vert k\vert<\frac 1 {\mu}$
and
$$ \Vert B^*u\Vert_{W_2}\le \mu  \Vert C^*u\Vert_{W_1}, \quad \forall u\in {\mathcal D}(A^{\frac 12}). $$
In \cite{KP} the coefficient $k$ is time dependent, as here, but it is assumed
that 
$$
k(t)=k_1(t)+k_2(t),
$$
with  $k_1\in L^1([0,+\infty))$, $k_2\in L^\infty ([0,+\infty))$, and
$||k_2||_\infty$ smaller than a suitable constant depending on the damping operator $CC^*.$

The time-dependent time delay case has been already considered in \cite{KP}. However,  here we work in a more general setting. Indeed, there, in addition to \eqref{taubar}, we assumed that the time delay function belongs to
$W^{1,\infty}(0, +\infty)$ and that
$\tau^\prime (t)\le c<1.$ 
These are the classical assumptions, commonly used to deal with wave-type equations with a time variable time delay (see e.g. \cite{NPV11, ChentoufMansouri, Feng} ). On the contrary, here, we only assume that the time delay is a continuous function bounded from above. Therefore, our results significantly improve previous related literature. 

Stability results in the  presence of time delay feedbacks have also been obtained for specific models, with  $k$ constant, mainly in the linear setting (see e.g. \cite{Ait, AG, AM, Chentouf, Dai, Oquendo, Said}). For recent results about Korteweg-de Vries-Burgers and  higher-order dispersive equations
with time delay see \cite{KP2} and \cite{Capistrano} respectively. We also refer to \cite{Ma} for a recent analysis of nonlinear damped Lam\'e systems with delay. 

The rest of the paper is organized as follows. In section \ref{prel} we precise our assumptions and rewrite system \eqref{modello} in an abstract form. Moreover, we give some preliminary estimates. In section \ref{exp}, first we give an exponential stability result under an appropriate well-posedness assumption; then, we show that the well-posedness assumption is satisfied for model \eqref{modello} and so the exponential decay estimate holds for {\em small} initial data. Section \ref{Examples} is devoted, as a concrete example for which the abstract theory is applicable, to the damped wave equation with a nonlinear source.
\section{Assumptions and preliminary estimates}
\label{prel} \hspace{5mm}

\setcounter{equation}{0}

In this section, we give some preliminary estimates which will be useful to prove our well-posedness and exponential stability results. First, we precise the assumptions on the nonlinear term of model \eqref{modello}.

Let $\psi : {\mathcal D}(A^{\frac 1 2})\rightarrow \RR$ be a functional having G\^{a}teaux derivative $D\psi(u)$ at every $u\in {\mathcal D}(A^{\frac 12}).$ In the spirit of \cite{ACS}, we assume the following hypotheses:
\begin{itemize}
\item[{(H1)}] For every $u\in {\mathcal D}(A^{\frac 1 2})$, there exists a constant $c(u)>0$ such that
$$
|D\psi(u)(v)|\leq c(u) ||v||_{{H}} \qquad \forall v\in {\mathcal D}(A^{\frac 1 2}).
$$
Then, $\psi$ can be extended to the whole  $H$ and
 we denote by $\nabla \psi(u)$ the unique vector representing $D\psi(u)$ in the Riesz isomorphism, i.e.
$$
\langle \nabla \psi(u), v \rangle_H =D\psi(u) (v), \qquad \forall v\in H;
$$
\item[ (H2)] for all $r>0$ there exists a constant $L(r)>0$ such that
$$
||\nabla \psi (u)-\nabla \psi (v)||_H \leq L(r) ||A^{\frac 12}(u-v)||_H,
$$
for all $u,v\in {\mathcal D}(A^{\frac 12})$ satisfying $||A^{\frac 12} u||_H\leq r$ and $||A^{\frac 12} v||_H\leq r$.
\item[{ (H3)}] $\psi(0)=0,$  $\nabla \psi(0)=0$ and
there exists a strictly increasing continuous function $h$ such that
\begin{equation}
\label{stima_h}
||\nabla \psi (u)||_H\leq h(||A^{\frac 12} u||_H)||A^{\frac 12}u||_H,
\end{equation}
for all $u\in {\mathcal D}(A^{\frac 12})$.
\end{itemize}

Now, we want to reformulate model \eqref{modello} in an abstract way. Let us introduce the Hilbert space
$$
{\mathcal H}=\mathcal{D}(A^{\frac 12}) \times H,
$$
endowed with the inner product
$$
\left\langle
\left (
\begin{array}{l}
u\\
v
\end{array}
\right ),
\left (
\begin{array}{l}
\tilde u\\
\tilde v
\end{array}
\right )
\right\rangle_{\W}:= \langle A^{\frac 12}u, A^{\frac 12} \tilde u\rangle_H+\langle v, \tilde v\rangle_H.
$$
If we denote $v(t)=u_t(t)$ and $U(t):=(u(t),v(t))^T$, we can rewrite system \eqref{modello} in the following abstract form
\begin{equation}\label{abstract_form}
\begin{array}{l}
\displaystyle{U'(t)=\mathcal AU(t)- k(t)\mathcal BU(t-\tau(t))+F(U(t)),}\\
\displaystyle{U(0)=U_0,}\\
\displaystyle{\mathcal BU(t)=f(t), \quad t\in[-\overline\tau,0],}
\end{array}
\end{equation}
where 
$$
\mathcal A=
\begin{pmatrix}
0 & 1 \\ 
-A & -CC^*
\end{pmatrix}, \quad 
\mathcal BU(t)= \left( \begin{array}{l} \hspace{0.7 cm}0 \\ BB^* v(t)\end{array} \right) \quad \text{and} \quad F(U(t))=\left( \begin{array}{l} \hspace{0.65 cm} 0 \\ \nabla \psi(u(t)) \end{array}\right).
$$
We know that, under controllability assumptions on the  damping operator $CC^*$ (see for instance \cite{Bardos, K}), $\mathcal A$ generates an exponentially stable $C_0$-semigroup $\{ S(t)\}_{t\geq 0}$, namely there exist $M,\omega >0$ such that
\begin{equation}\label{decay_semigroup}
||S(t)||_{\mathcal L(\mathcal H)} \leq Me^{-\omega t}, \quad \forall t\geq 0.
\end{equation}
Moreover, the previous hypotheses (H2)-(H3) on $\psi$ imply the following properties on $F$:
\begin{itemize}
\item[(F1)] $F(0)=0$;
\item[(F2)] for any $r>0$ there exists a constant $L(r)>0$ such that
$$
||F(U)-F(V)||_\W\leq L(r) ||U-V||_\W,
$$
whenever $||U||_\W,||V||_\W\leq r$.
\end{itemize}
Let us denote
\begin{equation}\label{normab}
\Vert B\Vert_{{\mathcal L}(W_2,H)}=\Vert B^*\Vert_{{\mathcal L}(H, W_2)}=b.
\end{equation}
Then, 
\begin{equation}\label{normabcal}
\Vert\mathcal B\Vert_{{\mathcal L}(\W)}=b^2.
\end{equation}
We will prove well-posedness and exponential stability for system \eqref{modello}, for small initial data, under the assumption

\begin{equation}\label{assumption_delay}
M  b^2e^{\omega\overline{\tau}} \int_0^t \vert k(s)\vert ds  \leq \gamma+\omega' t, \quad \forall \ t>0,
\end{equation}
for suitable constants $\gamma\geq 0$ and $\omega'\in[0,\omega).$ 

 Let us introduce the following energy functional associated to system \eqref{modello}:
\begin{equation}\label{energy}
E(t):=\frac 12 ||u_t(t)||_H^2+\frac 12 ||A^{\frac 12}u(t)||_H^2-\psi(u(t))+\frac 12  \int_{t-\overline\tau}^t \vert k(s)\vert \cdot ||B^* u_t(s)||_{W_2}^2 ds.
\end{equation}
Moreover, let us define the functional
\begin{equation}\label{secondo_funzionale}
{\mathcal E}(t):=\max \left \{  \frac 14 \max_{s\in [-\overline\tau,0] }\Vert g(s)\Vert^2_{W_2}, \ 
b^2\max_{s\in [0, t] }E(s)  \right \}.
\end{equation}

Then, in particular, for $t=0,$
\begin{equation}\label{funzionale_in_zero}
{\mathcal E}(0)=\max \left \{  \frac 14 \max_{s\in [-\overline\tau,0] }\Vert g(s)\Vert^2_{W_2}, \ 
b^2E(0)  \right \}.
\end{equation}

We have the following result.
\begin{Proposition}\label{proposition}
Let $u$ be a solution to \eqref{modello}. If $E(t)\geq \frac 14 ||u_t(t)||_H^2$ for any $t\geq 0$, then
\begin{equation}\label{tesi_prop}
E(t)\le   \bar C(t) {{\mathcal E}(0)}
\end{equation}
for any $t\ge \overline 0$, where 
\begin{equation}\label{Cbar}
\displaystyle{
\bar C (t):=\frac 1 {b^2} \,e^{6b^2\int_{0}^t \vert k(s)\vert  ds}}.
\end{equation}
\end{Proposition}
\begin{proof}
Differentiating the energy  $E$ in time, we have 
$$
\begin{array}{l}
\displaystyle{\frac{d}{dt}E(t)=\langle u_t(t),u_{tt}(t)\rangle_H+\langle A^{\frac 12} u(t),  A^{\frac 12} u_t(t)\rangle_H-\langle \nabla \psi(u(t)), u_t(t)\rangle }\\
\hspace{2 cm}
\displaystyle{+\frac 12 \vert k(t)\vert\cdot \Vert B^*u_t(t)\Vert^2_{W_2}-\frac 12  \vert k(t-\overline\tau)\vert \cdot \Vert B^*u_t(t-\overline \tau)||^2_{W_2}.}
\end{array}
$$
By using the equation \eqref{modello}, we deduce
$$
\begin{array}{l}
\displaystyle{\frac{d}{dt}E(t) =-||C^* u_t(t)||_{W_1}^2- k(t)\langle B^* u_t(t),B^*u_t(t-\tau(t))\rangle_{W_2}}\\
\hspace{2 cm}
\displaystyle{+\frac 12  \vert k(t)\vert \cdot \Vert B^*u_t(t)\Vert_{W_2}-\frac 12  \vert k(t-\overline\tau)\vert \cdot \Vert B^*u_t(t-\overline\tau)\Vert_{W_2}.}
\end{array}
$$
Using Young inequality in the second term in the previous identity, we get
\begin{equation}\label{uffa}
\begin{array}{l}
\displaystyle{ \frac{d}{dt}E(t) \le  \vert k(t)\vert \Vert B^* u_t(t)\Vert_{W_2}^2+\frac 12\vert k(t)\vert \Vert B^* u_t(t-\tau(t))\Vert_{W_2}^2}\\
\hspace{2 cm}
\displaystyle{\le \frac 32  \vert k(t)\vert \, \max_{s\in [t-\overline t, t]}\{ \Vert B^*u_t(s)\Vert_{W_2}^2\}, \quad t\ge 0.}
\end{array}
\end{equation}
Now, observe that, for $t\in [0, \overline\tau),$ or
$$\max_{s\in [t-\overline t, t]}\{ \Vert B^*u_t(s)\Vert_{W_2}^2\}=\max_{s\in [t-\overline t, 0]}\{ \Vert B^*u_t(s)\Vert_{W_2}^2\}\le \max_{s\in [-\overline t, 0]}\{ \Vert g(s)\Vert_{W_2}^2\}\le 4{\mathcal E}(t),$$
or

$$\max_{s\in [t-\overline t, t]}\{ \Vert B^*u_t(s)\Vert_{W_2}^2\}=\max_{s\in [0, t]}\{ \Vert B^*u_t(s)\Vert_{W_2}^2\}\le 4b^2\max_{s\in [0, t]}E(s)\le4{\mathcal E}(t),$$
where we used the assumption  $E(t)\geq \frac 14 \Vert u_t(t)\Vert_H^2$ for any $t\geq 0.$
For $t\ge\overline\tau,$ from $E(t)\geq \frac 14 \Vert u_t(t)\Vert_H^2,$ we have

$$\max_{s\in [t-\overline t, t]}\{ \Vert B^*u_t(s)\Vert_{W_2}^2\}\le 4b^2\max_{s\in [t-\overline t, t]}E(s)\le4{\mathcal E}(t).$$
Then, using the above estimates in \eqref{uffa}, we deduce

$$
\frac{d}{dt}E(t) \le  6   \vert k(t)\vert {\mathcal E}(t), \quad t\ge 0.
$$ 

Now, observe that the function ${\mathcal E}(t)$ is constant or it increases as $b^2E(t).$ 
Then, 
$$
\frac{d}{dt} {\mathcal E}(t)\le 6 b^2  \vert k(t)\vert {\mathcal E}(t), \quad t\ge 0,
$$ 
 and Gronwall's inequality concludes the proof.
\end{proof}

Before proving our well-posedness and exponential stability results, we need some preliminary estimates.
Assume, for the moment, that the time delay is bounded from below by a positive constant, namely

\begin{equation}\label{bound_below}
\tau(t)\ge \tau_0, \quad \forall t\ge 0,
\end{equation}
for some positive constant $\tau_0.$ We will remove such an assumption later on.

We have the following local well--posedness result.
\begin{Lemma}
\label{lemma1}
Assume \eqref{bound_below}.
Let us consider the system \eqref{abstract_form} with initial data $U_0\in {\mathcal H}$ and  $f\in C([-\overline\tau,0]; {\mathcal H}).$ Then, there exists a unique continuous local solution $U(\cdot)$ defined on a time  interval $[0,\delta)$, with $\delta \le\tau_0.$ 
\end{Lemma}
\begin{proof}
In $[0,\tau_0]$, we can rewrite the abstract system \eqref{abstract_form} as  problem without any time delays:
\begin{eqnarray*}
U'(t)&=& \mathcal AU(t)-k(t) f(t-\tau(t))+F(U(t)), \quad t\in (0, \tau_0),\\
U(0)&=&U_0.
\end{eqnarray*}
Then, the standard theory of nonlinear semigroups (see e.g. \cite{Pazy}) ensures the existence of a unique solution  on a set $[0,\delta)$, with $\delta \le\tau_0$.
\end{proof}
The next two lemmas will allow to extend the local solution when the initial energy is sufficiently small.
\begin{Lemma}
\label{lemma2}
Assume \eqref{bound_below}.
Let $U(\cdot)$ be a non-zero solution to \eqref{abstract_form} defined on the interval $[0, \delta),$ and let $h$ be the strictly increasing function appearing in \eqref{stima_h}. 
 If $h(||A^\frac{1}{2} u_0||_H)<\frac 12,$ then $E(0)>0.$
\end{Lemma}
\begin{proof}
First of all observe that, from the assumption (H3) on $\psi,$ we can deduce
\begin{equation}
\label{assumptionPsi}
\begin{array}{l}
\displaystyle{|\psi(u)|\leq \int_0^1 |\langle \nabla \psi (su),u\rangle | ds} \\
\hspace{1,15 cm}
\displaystyle{\leq  ||A^\frac{1}{2}u||^2_H \int_0^1 h(s||A^\frac{1}{2}u||_H)s ds\leq \frac{1}{2} h(||A^\frac{1}{2}u||_H)||A^\frac{1}{2}u||^2_H.}
\end{array}
\end{equation}
Hence, under the assumption $h (\Vert A^{\frac 12} u_0\Vert_H) < \frac {1} 2,$ we have that
\begin{equation}\label{27luglio}
\begin{array}{l}
\displaystyle{ E(0)=\frac{1}{2}||u_1||_H^2+\frac{1}{2}||A^\frac{1}{2}u_0||_H^2-\psi(u_0)+\frac{1}{2}\int_{-\overline\tau}^0 |k(s)|\cdot ||B^*u_t(s)||^2_{W_2} ds}\\
\hspace{0,9 cm}
\displaystyle{ >\frac{1}{4}||u_1||^2_H+\frac{1}{4}||A^\frac{1}{2}u_0||^2_H +\frac{1}{4} \int_{-\overline\tau}^0 |k(s)| \cdot ||B^*u_t(s)||^2_{W_2} ds. }\\
\end{array}
\end{equation}
Then, since $U$ is a non-zero solution, the right-hand side of previous inequality is strictly positive. Therefore, the claim  is proven.
\end{proof}
\begin{Lemma}
\label{lemmaspezzato}
Assume \eqref{bound_below}.
Let $U(\cdot)$ be a non-zero solution to \eqref{abstract_form} defined on the interval $[0, \delta),$ with $\delta\le T.$ Let $h$ be the strictly increasing function appearing in \eqref{stima_h}. 
If $h(||A^\frac{1}{2} u_0||_H)<\frac 12$ and
$
h \left( 2 \bar{C}^\frac{1}{2}(T) {\mathcal E}^\frac{1}{2}(0) \right) <\frac 12,
$
with $\bar{C}(\cdot)$ defined as in \eqref{Cbar}, then
\begin{equation}
\label{stima E dal basso}
\begin{array}{l}
\displaystyle{ E(t)>\frac{1}{4}||u_t(t)||_H^2+\frac{1}{4}||A^\frac{1}{2}u(t)||_H^2 +\frac{1}{4}\int_{t-\overline\tau}^t |k(s)| \cdot ||B^*u_t(s)||_{W_2}^2 ds}
\end{array}
\end{equation}
for all $t\in[0, \delta)$. In particular,
\begin{equation}\label{J2}
E(t)>  \frac 14 \Vert U(t)\Vert_{\W}^2, \quad \mbox{for all} \ \ t\in [0, \delta).
\end{equation}
\end{Lemma}
\begin{proof}
To prove  \eqref{stima E dal basso}, we argue by contradiction. Let us denote
$$
r:=\sup \{ s\in [0,\delta) : \eqref{stima E dal basso} \quad \text{holds} \quad \forall t\in [0,s)\}.
$$
We suppose by contradiction that $r<\delta$. Then, by continuity, we have
\begin{equation}
\label{continuita}
\begin{array}{l}
\displaystyle{E(r)=\frac{1}{4}||u_t(r)||^2_H+\frac{1}{4}||A^\frac{1}{2}u(r)||_H^2+\frac{1}{4}\int_{r-\overline\tau}^r |k(s)| \cdot ||B^*u_t(s)||_{W_2}^2 ds.}
\end{array}
\end{equation}
Now, since from \eqref{continuita}
$$
\frac{1}{4} \Vert A^{\frac 1 2} u(r)\Vert^2_H\leq E(r),
$$
by using Proposition \ref{proposition}, we deduce that
\begin{equation}\label{risultato}
\begin{array}{l}
\displaystyle{ h(||A^\frac{1}{2}u(r)||_H)\leq h\left( 2 E^\frac{1}{2}(r)\right)  \leq h\left( 2\bar{C}^\frac{1}{2}(T){\mathcal E}^\frac{1}{2}(0)\right) <\frac{1}{2}.}
\end{array}
\end{equation}
Then, we have that
$$
\begin{array}{l}
\displaystyle{ E(r)=
\frac{1}{2}||u_t(r)||_H^2+\frac{1}{2}||A^\frac{1}{2}u(r)||_H^2-\psi(u(r))+\frac{1}{2}\int_{r-\overline\tau}^r|k(s)|\cdot ||B^*u_t(s)||^2_{W_2} ds}\\
\hspace{0.9 cm}
 \displaystyle{
>\frac{1}{4}||u_t(r)||_H^2+\frac{1}{4}||A^\frac{1}{2}u(r)||_H^2+\frac{1}{4}\int_{r-\overline\tau}^r|k(s)| \cdot ||B^*u_t(s)||_{W_2}^2 ds,}
\end{array}
$$
where in the last estimate we used \eqref{assumptionPsi} and \eqref{risultato}. This contradicts the maximality of $r$. Hence, $r=\delta$ and this concludes the proof of the lemma.
\end{proof}

\section{Exponential stability}
\label{exp}\hspace{5mm}

\setcounter{equation}{0}

We can now proceed to prove the well-posedness assumption for system \eqref{modello} under the technical assumption (that we will remove later on) \eqref{bound_below} on the time delay function.

First, we give an exponential decay result for the abstract model \eqref{abstract_form} under a suitable well-posedness assumption.

\begin{Theorem}\label{generaleCV}
Assume \eqref{assumption_delay}.
Moreover, suppose that 
\begin{itemize}
\item[{(I)}] there exist $\rho>0$, $C_\rho>0$,  with $L(C_\rho)<\frac{\omega-\omega '}{M}$ such that if $U_0 \in \W$ and if $f\in C([-\overline{\tau} ,0];\W)$ satisfy
\begin{equation}\label{well-posedness}
||U_0||^2_{\W}+ \int_{-\overline{\tau}}^0 |k(s)| \cdot ||f(s)||^2_{\W} ds <\rho ^2,
\end{equation}
then the system \eqref{abstract_form} has a unique solution $U\in C([0,+\infty);\W)$ satisfying $||U(t)||_{\W}\leq C_\rho$ for all $t>0$.
\end{itemize}
Then, for every solution $U$ of \eqref{abstract_form}, with initial data $(U_0, f)$ satisfying \eqref{well-posedness},
\begin{equation}
\label{stimaesponenziale}
||U(t)||_{\W}\leq Me^\gamma \left (\Vert U_0\Vert_{\mathcal H}+e^{\omega \overline\tau}K \max_{s\in[-\overline\tau, 0]}\left\{\Vert e^{\omega s} f(s)\Vert_{\mathcal H}\right\}\right )e^{-(\omega -\omega'-ML(C_{\rho}))t}, 
\end{equation}
for any $t\geq 0$.
\end{Theorem}
\begin{proof}
From Duhamel's formula we have
$$
\begin{array}{l}
\displaystyle{||U(t)||_{\mathcal H}\leq Me^{-\omega t} ||U_0||_{\mathcal H}+Me^{-\omega t} \int_0^t e^{\omega s}|k(s)| \cdot ||\mathcal BU(s-\tau(s))||_{\mathcal H} ds}\\
\hspace{1.7 cm}
\displaystyle{ +ML(C_\rho)e^{-\omega t} \int_0^t e^{\omega s} ||U(s)||_{\mathcal H} ds,}
\end{array}
$$
where we have used the fact that $||F(U(t))||_{\mathcal H}\leq L(C_\rho) ||U(t)||_{\mathcal H}$ for any $t\geq 0$. Then, we deduce that
$$
\begin{array}{l}
\displaystyle{ ||U(t)||_{\mathcal H}\le Me^{-\omega t} ||U_0||_{\mathcal H}+Me^{-\omega t} \int_0^{\overline\tau} e^{\omega s} |k(s)| \cdot ||{\mathcal B}U(s-\tau(s))||_{\mathcal H} ds}\\
\hspace{1.5 cm}
\displaystyle{+ Me^{-\omega t} \int_{\overline\tau}^t e^{\omega s} b^2 |k(s)|\cdot ||U(s-\tau(s))||_{\mathcal H}  ds +ML(C_\rho)e^{-\omega t} \int_0^t e^{\omega s} ||U(s)||_{\mathcal H} ds}\\
\hspace{1.2 cm}
\displaystyle{
\le Me^{-\omega t} ||U_0||_{\mathcal H}+Me^{-\omega t}e^{\omega\overline\tau} \int_0^{\overline\tau} e^{\omega (s-\tau(s))} |k(s)| \cdot ||{\mathcal B}U(s-\tau(s))||_{\mathcal H} ds}\\
\hspace{1.5 cm}
\displaystyle{+ Me^{-\omega t}e^{\omega\overline\tau} \int_{\overline\tau}^t e^{\omega (s-\tau(s))} b^2 |k(s)|\cdot ||U(s-\tau(s))||_{\mathcal H}  ds +ML(C_\rho)e^{-\omega t} \int_0^t e^{\omega s} ||U(s)||_{\mathcal H} ds.}
\end{array}
$$
Now, observe that
\begin{equation}\label{March16_4}
\begin{array}{l}
\displaystyle{
\int_0^{\overline\tau} e^{\omega (s-\tau(s))}\vert k(s)\vert\cdot \Vert {\mathcal B}U(s-\tau(s))\Vert ds}\\
\displaystyle{\hspace{1 cm}
\le
\int_0^{\overline\tau} \vert k(s)\vert\left (   \max_{s\in[-\overline\tau, 0]}\left\{ e^{\omega s}\Vert f(s)\Vert_{\mathcal H}\right \}+b^2\max_{r\in[0,s]}  \left\{e^{\omega r}\Vert U(r)\Vert_{\mathcal H}  \right\}    \right) ds}\\
\hspace{1 cm}\le\displaystyle{ K \max_{s\in[-\overline\tau, 0]}\left\{e^{\omega s}\Vert f(s)\Vert_{\mathcal H}\right\}+\int_0^{\overline\tau}b^2\vert k(s)\vert \max_{r\in[0,s]}\left\{e^{\omega r}\Vert U(r)\Vert_{\mathcal H}\right\} ds.
}
\end{array}
\end{equation}
Therefore, using \eqref{March16_4} in the previous inequality, we deduce
$$
\begin{array}{l}
\displaystyle{ ||U(t)||_{\mathcal H}\le Me^{-\omega t} \left (||U_0||_{\mathcal H}+
e^{\omega\overline\tau}  K \max_{s\in[-\overline\tau, 0]}\left\{e^{\omega s}\Vert f(s)\Vert_{\mathcal H}\right\}
\right )
}\\
\hspace{1,5 cm}
\displaystyle{+ Me^{-\omega t} e^{\omega\overline\tau}\int_{0}^t  b^2 |k(s)|\max_{r\in [s-\overline\tau, s]\cap [0,s]} \left\{e^{\omega r}||U(r)||_{\mathcal H}\right\}  ds}\\
\hspace{1,5 cm}\displaystyle{ +ML(C_\rho)e^{-\omega t} \int_0^t  \max_{r\in [s-\overline\tau, s]\cap [0,s]} \left\{e^{\omega r} ||U(r)||_{\mathcal H}\right\} ds.}
\end{array}
$$

Then,
$$
\begin{array}{l}
\displaystyle{ e^{\omega t}||U(t)||_{\mathcal H} \le M \left (||U_0||_{\mathcal H}+
e^{\omega\overline\tau} K \max_{s\in[-\overline\tau, 0]}\left\{\Vert e^{\omega s} f(s)\Vert_{\mathcal H}\right\}
\right )}\\
\hspace{2 cm}
\displaystyle{+M  \int_0^t(b^2 e^{\omega\overline\tau} \vert k(s)\vert +L(C_\rho))  \max_{r\in [s-\overline\tau, s]\cap [0,s]}\left\{e^{\omega r} ||U(r)||_{\mathcal H} \right \}ds, \quad t\ge 0.}
\end{array}
$$
Therefore, it is easy to see that 
$$
\begin{array}{l}
\displaystyle{ \max_{s\in [t-\overline\tau, t]\cap [0,t]}\left\{e^{\omega s}||U(s)||_{\mathcal H}\right\} \le M  \left (||U_0||_{\mathcal H}+
e^{\omega\overline\tau} K \max_{s\in [-\overline\tau, 0]}\left\{\Vert e^{\omega s} f(s)\Vert_{\mathcal H}\right\}
\right )}\\
\hspace{2 cm}
\displaystyle{+M  \int_0^t(b^2e^{\omega\overline\tau}  \vert k(s)\vert +L(C_\rho))  \max_{r\in [s-\overline\tau, s]\cap [0,s]}\left\{e^{\omega r} ||U(r)||_{\mathcal H} \right \}ds, \quad t\ge 0.}
\end{array}
$$

Hence, if we denote 
$$\tilde u (t):=\max_{s\in [t-\overline\tau, t]\cap [0,t]}\left\{e^{\omega s}||U(s)||_{\mathcal H}\right\},$$
 Gronwall's estimate implies
$${\tilde u}(t)\le {\tilde M} e^{  Mb^2 e^{\omega\overline\tau}\int_0^t \vert k(s)\vert ds +ML(C_\rho) t},
$$
where
$$\tilde M:= M\left (
\Vert U_0\Vert_{\mathcal H}+e^{\omega \overline\tau}K \max_{s\in[-\overline\tau, 0]}\left\{\Vert e^{\omega s} f(s)\Vert_{\mathcal H}\right\}
\right ). 
$$
Then, 
$$e^{\omega t}\Vert U(t)\Vert_{\mathcal H}\le {\tilde M} e^{  Mb^2 e^{\omega\overline\tau}\int_0^t \vert k(s)\vert ds +ML(C_\rho) t},
$$
and, by assumption \eqref{assumption_delay}, we get the exponential  decay estimate \eqref{stimaesponenziale}.
\end{proof}
Now, from Theorem \ref{generaleCV}, in order to have the exponential stability of solutions to \eqref{modello}, we need to show that the well-posedness assumption $(I)$ actually holds true for system \eqref{modello}.

\begin{Theorem}\label{wellpos_thm}
Assume \eqref{assumption_delay}. Tthere exists a constant  $\rho>0$ such that if
$$
||u_1||_H^2+||A^{\frac 12}u_0||_H^2 +\int_{-\overline\tau}^0 |k(s)| \cdot \Vert g(s)\Vert_{W_2}^2 ds \le \rho^2, 
$$ 
and
$$
\max_{s\in [-\overline\tau,  0]}\Vert g(s)\Vert_H\le 2b\rho,$$
then the solution $u$ of
\eqref{modello}, corresponding to the initial data $(u_0, u_1)\in {\mathcal H}$ and $g\in C([-\overline\tau,0]; {W_2}),$  satisfies the exponential decay estimate
$$E(t)\le \tilde K e^{-\tilde\beta t}, \quad t\ge 0,$$
where $\tilde K$ is a constant dependent on the initial data and $\tilde\beta>0.$
\end{Theorem}
\begin{proof}
{ Part I)} First, we assume \eqref{bound_below}. This will allow us to prove the theorem by applying previous lemmas. Then, in the second part of the proof, we will remove this technical assumption proving the general result.
 
Let us fix a time $T>0$ sufficiently large,
such that
\begin{equation}\label{stimaN}
\begin{array}{l}
\displaystyle{C_T:=2M^2e^{2\gamma} \max \left\{b^2(1+K   b^2e^{\omega \overline\tau}), \frac  1 4 e^{\omega \overline\tau}\right\} \left(1+4K^2b^2e^{2\omega\overline \tau}\right) e^{-(\omega-\omega')T}\le 1.}
\end{array}
\end{equation}
Moreover, let $\rho>0$ be such that
$$
\rho\leq \frac{1}{2b{\bar C }^{\frac 12}(T)}h^{-1}\left( \frac 12\right),
$$
where $\bar C(\cdot)$ is the increasing function 
defined in \eqref{Cbar}.
Now, consider initial data such that
\begin{equation}\label{dis1}
||u_1||_H^2+||A^{\frac 12}u_0||_H^2 +\int_{-\overline\tau}^0 |k(s)| \cdot \Vert g(s)\Vert_{W_2}^2 ds \le \rho^2, 
\end{equation}
and
\begin{equation}\label{March16_1}
\max_{s\in [-\overline\tau,  0]}\Vert g(s)\Vert_H\le 2b\rho.
\end{equation}

We observe that this implies (by referring to the abstract formulation \eqref{abstract_form})
\begin{equation}\label{March16_7}
||U_0||_{\W}^2 +\int_{-\overline\tau}^{0} |k(s)| \cdot \Vert f(s)\Vert_{\mathcal H}^2 ds \le b^2\rho^2
\end{equation}
and 
\begin{equation}\label{March16_8}
\max_{s\in [-\overline\tau,  0]}\Vert f(s)\Vert_{\mathcal H}\le 2 b^2\rho.
\end{equation}
First of all, from Lemma \ref{lemma1} we know that there exists a local solution $u$ to \eqref{modello} on a time interval $[0,\delta).$  From our assumption on the initial data, we have that
$$
h(||A^{\frac 12} u_0||_H)\le h(\rho) \le h\left( \frac{1}{2b\bar C ^{\frac 12} (T)} h^{-1}\left( \frac 12\right)\right) <\frac 12,
$$
where we have used the fact that $b^2\bar C (T) >1$. 
 Hence, by Lemma \ref{lemma2}, $E(0)>0$. Moreover, from \eqref{assumptionPsi} we get
$$
\begin{array}{l}
\displaystyle{E(0)\leq \frac 12 ||u_1||_H^2+\frac 34 ||A^{\frac 12} u_0||_H^2+\frac 12 \int_{-\overline\tau}^0 \vert k(s)\vert \cdot \Vert g(s)\Vert_{W_2}^2 ds\le \rho^2,}
\end{array}
$$
which implies
\begin{equation}\label{March16_2}
h\left(2\bar C^{\frac 12} (T)bE^{\frac 12} (0)\right)< h\left(2b\bar C^{\frac 12} (T) \rho\right) <h\left(h^{-1}\left(\frac 12\right)\right) =\frac 12.
\end{equation}
Moreover, assumption 
\eqref{March16_1} implies
\begin{equation}\label{March16_3}
h\left(\bar C^{\frac 12} (T)\max_{s\in [-\overline t, 0]}\Vert g(s)\Vert_H\right)< h\left(2b\bar C^{\frac 12} (T) \rho\right) <h\left(h^{-1}\left(\frac 12\right)\right) =\frac 12.
\end{equation}
From \eqref{March16_2} and \eqref{March16_3}, recalling \eqref{funzionale_in_zero}, we deduce

\begin{equation}\label{March16_5}
h\left(2\bar C^{\frac 12} (T){\mathcal E}^{\frac 12} (0)\right)< \frac 12.
\end{equation}

Hence, we can apply Lemma \ref{lemmaspezzato} and we can infer that \eqref{J2} is satisfied for all $t\in[0,\delta)$. Then, we can use Proposition \ref{proposition} getting
\begin{equation}\label{29luglio}
\begin{array}{l}
\displaystyle{0<\frac 14 ||u_t(t)||_H^2+\frac 14 ||A^{\frac 12} u(t)||_H^2 }\\
\hspace{0.4 cm}
\displaystyle{+\frac 14  \int_{t-\overline\tau}^t |k(s)| \cdot ||B^*u_t(s)||_{W_2}^2 ds< E(t)\leq  \bar C(T){\mathcal E}(0),}
\end{array}
\end{equation}
for any $t\in [0,\delta)$. Then, we can extend the solution in $t=\delta$ and on the whole interval $[0, T]$.  
Now, for $t=T$ we have that
$$
\begin{array}{l}
\displaystyle{h(||A^{\frac 12} u(T)||_H) \le h(2E^{\frac 12} (T))}\\
\hspace{3 cm}
\displaystyle{\le h(2\bar C ^{\frac 12} (T){\mathcal E}^{\frac 12} (0))<\frac 12.}
\end{array}
$$
Moreover, from \eqref{funzionale_in_zero} and from the assumptions  \eqref{dis1} and \eqref{March16_1} on the initial data, we deduce
$$
\frac 14 ||U(t)||_{\W}^2\le E(t)\le \bar C (T ){\mathcal E(0)}<\bar C(T)b^2\rho^2, \quad \forall \ t \in [0,T],
$$
and so 
$$
||U(t)||_{\W}\le C_{b\rho}:= 2\bar C^{\frac 12} (T)b\rho,
$$
for any $t\in[0, T]$. Now, eventually choosing a smaller values of $\rho$, we suppose that $\rho$ is such that $L(C_{b\rho})<\frac{\omega-\omega'}{2M}$. Therefore, the well-posedness assumption $(I)$ of Theorem \ref{generaleCV} is satisfied in the interval $[0, T]$. Hence, from Theorem \ref{generaleCV} we deduce the following estimate:

\begin{equation}\label{1}
||U(t)||_{\W}\le Me^{\gamma}  \left ( \Vert U_0\Vert_{\mathcal H}+e^{\omega \overline\tau}K \max_{s\in[-\overline\tau, 0]}\left\{\Vert e^{\omega s} f(s)\Vert_{\mathcal H}\right\} 
 \right )e^{-\frac{\omega -\omega'}{2}t},
\end{equation}
for any $t\in [0, T]$.
Recalling \eqref{March16_8}, from \eqref{1} we get
\begin{equation}\label{March16_9}
||U(t)||_{\W}\le Me^{\gamma}  \left ( \Vert U_0\Vert_{\mathcal H}+e^{\omega \overline\tau}2Kb^2 \rho 
 \right )e^{-\frac{\omega -\omega'}{2}t},
\end{equation}
Therefore, from \eqref{March16_9}, recalling \eqref{March16_7}, we get
\begin{equation}\label{March16_10}
\begin{array}{l}
\displaystyle{
||U(t)||_{\mathcal H}^2\le 2M^2e^{2\gamma}\left( \Vert U_0\Vert_{\mathcal H}^2+4e^{2\omega \overline\tau}  K^2b^4\rho^2\right) e^{-(\omega-\omega')t}
}\\
\hspace{2 cm}\le \displaystyle{
2M^2e^{2\gamma}b^2\rho^2 \left (1+4K^2b^2e^{2\omega \overline\tau} \right )e^{-(\omega-\omega')t}, \quad \forall\ t\in [0, T].
}
\end{array}
\end{equation}
Moreover, from \eqref{March16_10}, we can deduce
\begin{equation}\label{March16_11}
\int_{T-\overline\tau}^T \vert k(s)\vert \cdot\Vert B^*u_t(s)\Vert_{W_2}^2 ds
\le 2M^2Ke^{2\gamma}e^{\omega\overline\tau}b^4\rho^2 \left (1+4K^2b^2e^{2\omega \overline\tau} \right )e^{-(\omega-\omega')T}.
\end{equation}
Then, from \eqref{March16_10} and \eqref{March16_11}, we obtain
 
$$
\begin{array}{l}
\displaystyle{||U(T)||_\W^2 + \int_{T-\overline\tau}^{T}  |k(s)|\cdot ||B^* u_t(s)||_{W_2}^2 ds\le  C_T\rho^2\leq \rho^2,}
\end{array}
$$
where we have used \eqref{stimaN}. 
We can also check that
$$\max_{s\in [T-\overline\tau, T]} \Vert B^*u_t(s)\Vert_{W_2}^2
\le  2M^2e^{2\gamma}e^{\omega\overline\tau}b^4\rho^2 \left (1+4K^2b^2e^{2\omega \overline\tau} \right )e^{-(\omega-\omega')T} \le 4 C_Tb^2\rho^2\le 4b^2\rho^2
$$
and so
$$\max_{s\in [T-\overline\tau, T]} \Vert B^*u_t(s)\Vert_{W_2}\le 2b\rho.$$
These last two conditions (cf. \eqref{March16_7} and \eqref{March16_8}) allow us 
to
proceed by applying  analogous arguments than before on the interval $[T, 2T],$ obtaining a solution on the interval $[0,2T]$. Iterating the process, we find a unique global solution to \eqref{abstract_form} satisfying the well-posedness assumption $(I)$. Hence, under the assumption \eqref{bound_below}, the theorem is proved. 

 {Part II)}  Now, we want to remove  the technical assumption  \eqref{bound_below}. So, let $\tau(t)$ a time delay function satisfying 
$$0\le\tau(t)\le\overline\tau.$$
Let $\{\epsilon_n\}_n$ be a sequence of positive constants such that $\epsilon_n\rightarrow 0$ as $n\rightarrow +\infty.$ We can suppose $\epsilon_n\le 1,$ for all $n\in \nat.$
Let us consider the time delay functions
\begin{equation}\label{taun}
\tau_n(t)=\tau(t)+\epsilon_n, \quad n\in \nat.
\end{equation}
From \eqref{taun}, we have
\begin{equation}\label{tauuniform}
\epsilon_n\le \tau_n(t)\le\overline\tau +1,\quad \forall t\in [0, +\infty),
\end{equation}
and then, in particular, $\tau_n(t)$ satisfies the assumption \eqref{bound_below}. Therefore, the model \eqref{abstract_form}, 
with the same (sufficiently small) initial data, and time delay $\tau_n(t),$ has a unique solution $U_n\in C([0,+\infty); {\mathcal H})$ satisfying the exponential decay estimate of Theorem \ref{generaleCV}.

Applying the Duhamel formula, we can write, $\forall\ n\in\nat,$
\begin{equation}\label{D1}
U_n(t)=S(t)U_0+\int_0^t S(t-s)[F(U_n(s))-k(s){\mathcal B}U_n(s-\tau_n(s))] ds
\end{equation}

Hence, from \eqref{D1}, $\forall\ i,j\in \nat,$ we have

\begin{equation}\label{D2}
\begin{array}{l}
\displaystyle{
U_i(t)-U_j(t)=\int_0^t S(t-s)[F(U_i(s)) -F(U_j(s))] ds}\\
\hspace{2cm}\displaystyle{
 -\int_0^tS(t-s)k(s) [{\mathcal B}U_i(s-\tau_i(s))- {\mathcal B}U_j(s-\tau_j(s))] ds}.
\end{array}
\end{equation}
From the local Lipschitz continuity of $F$ (note that the functions $U_n$ are uniformly bounded in ${\mathcal H}$) and the estimate \eqref{decay_semigroup} on the semigroup
$\left\{S(t)\right\}_{t\ge 0},$ we deduce
\begin{equation}\label{D3}
\begin{array}{l}
\displaystyle{
\Vert U_i(t)-U_j(t)\Vert_{\mathcal H}\le Ce^{-\omega t}\left\{\int_0^t e^{\omega s} \Vert U_i(s) -U_j(s)\Vert_{\mathcal H} ds\right.}\\
\hspace{2cm}\displaystyle{
 \int_0^te^{\omega s}\vert k(s)\vert  \cdot \Vert U_i(s-\tau_i(s))- U_j(s-\tau_i(s))\Vert_{\mathcal H} ds}\\
\hspace{2cm}\displaystyle{\left.
 +\int_0^t e^{\omega s} \vert k(s)\vert \cdot \Vert U_j(s-\tau_i(s))- U_j(s-\tau_j(s))\Vert_{\mathcal H} ds\right\}}\\
\hspace{4 cm}\displaystyle{=Ce^{-\omega t}\{I_1+I_2+I_3\}}.
\end{array}
\end{equation}
Now, observe that, since $U_j\in C([0, +\infty; {\mathcal H})$ is locally uniformly continuous, $\forall \ j\in\nat,$ and
$\tau_i(t)-\tau_j(t)= \epsilon_i-\epsilon_j, \forall i,j\in\nat,$
then, for every fixed time $T>0,$ we can estimate
$$I_3\le {\tilde C}(T; \vert \epsilon_i-\epsilon_j\vert),$$
where the constant $\tilde C(T; \vert \epsilon_i-\epsilon_j\vert)$ tends to zero as $\epsilon_i-\epsilon_j\rightarrow 0.$

Using this last estimate in \eqref{D3}, we obtain  
\begin{equation}\label{D4}
\begin{array}{l}
\displaystyle{
\Vert U_i(t)-U_j(t)\Vert_{\mathcal H}\le Ce^{-\omega t}\Big \{\int_0^t e^{\omega s} \Vert U_i(s) -U_j(s)\Vert_{\mathcal H} ds}\\
\hspace{2cm}\displaystyle{
 \int_0^te^{\omega s}\vert k(s)\vert  \cdot \Vert U_i(s-\tau_i(s))- U_j(s-\tau_i(s))\Vert_{\mathcal H} ds\Big\}}\\
\hspace{3cm}\displaystyle{
 +C e^{-\omega t}\tilde C(T; \vert \epsilon_i-\epsilon_j\vert)}.
\end{array}
\end{equation}
Then, for each fixed $T>0,$ combining arguments as in the proof of Theorem \ref{generaleCV} to deal with the delayed arguments and the use of  Gronwall's Lemma, we conclude that for $n\rightarrow +\infty$ ($\epsilon_n\rightarrow 0$) the functions $U_n(\cdot)$ converge locally uniformly to a function $U\in C([0,+\infty); {\mathcal H})$ which is a solution of \eqref{abstract_form}. Moreover, since the time delay functions $\tau_n(\cdot)$ are uniformly bounded from above (see \eqref{tauuniform}), all the estimates on the functions $U_n$ are uniform with respect to $n\in\nat.$ 

Then, 
$U$
satisfies the well-posedness assumption of Theorem \ref{generaleCV} and so the exponential estimate \eqref{stimaesponenziale}. \end{proof}

\section{The semilinear damped wave equation}
\label{Examples}\hspace{5mm}

\setcounter{equation}{0}
In this section we consider, as  a concrete example for which previous abstract well-posedness and stability results hold, the damped wave equation with a source term. 
Let $\Omega$ be an open bounded subset of $\RR^d$, with boundary $\partial\Omega$ of class $C^2,$ and let $\mathcal O \subset \Omega$ be an open subset which satisfies the geometrical control property in \cite{Bardos}. 
For instance, $\mathcal O \subset \Omega$ can be a neighborhood of  the whole boundary $\partial\Omega$ or, denoting by $m$ the standard multiplier $m(x)=x-x_0,$ $x_0\in \RR^d,$ as in \cite{Lions}, $\mathcal O$ can be the intersection of $\Omega$ with an open neighborhood of the set 
$$\Gamma_0=\left\{ \,x\in\Gamma\ :\ m(x)\cdot \nu(x)>0\,\right\}.$$
 Moreover, let $\tilde{\mathcal{O}}\subset\Omega$ be another open subset. Denoting by $\chi_ {\mathcal O}$ and $\chi_ {\tilde {\mathcal O}}$  the characteristic functions of the sets ${\mathcal O}$ and $\tilde{{\mathcal O}}$ respectively,  we consider the following wave equation
\begin{equation}\label{wave}
\begin{array}{l}
\displaystyle{u_{tt}(x,t)-\Delta u(x,t)+a\chi_{\mathcal{O}}(x) u_t(x,t)+k(t)\chi_{\tilde{\mathcal O}}(x)u_t(x,t-\tau(t))}\\
\hspace{7 cm}
\displaystyle{=u(x,t)|u(x,t)|^\beta, \quad (x,t)\in\Omega\times (0,+\infty),}\\
\displaystyle{u(x,t)=0, \quad (x,t)\in\partial\Omega\times (0,+\infty),}\\
\displaystyle{u(x,0)=u_0(x), \quad u_t(x,0)=u_1(x), \qquad x\in\Omega,}\\
\displaystyle{u_t(x,s)=g(x,s) \quad (x,s)\in {\Omega}\times [-\overline\tau,0],}
\end{array}
\end{equation}
where $a$ is a positive constant, $\tau(t)$ is the time delay function satisfying $0\le\tau(t)\le\overline\tau,$ $\beta >0,$ and the delayed damping coefficient $k(\cdot):[-\overline\tau,+\infty)\to (0,+\infty)$ is a $L^1_{loc}([-\overline\tau ,+\infty))$ function satisfying \eqref{damp_coeff}. The system \eqref{wave} can be rewritten  in the form \eqref{modello} with
$H=L^2(\Omega),$ and the operator $A: {\mathcal D}(A)\rightarrow H$ defined as
$A=-\Delta$ with domain ${\mathcal D}(A)=H^2(\Omega)\cap H_0^1(\Omega)$. 
The operator $A$ is positive, self-adjoint, with dense domain and compact inverse in $H.$
We then define $W_1:=L^2({\mathcal O}),  W_2:=L^2(\tilde{\mathcal O}),$ and the operators 
$$C:W_1\rightarrow H: \quad v\rightarrow \sqrt{a} \tilde v\chi_ {\mathcal O},$$
$$B:W_2\rightarrow H: \quad v\rightarrow \tilde v\chi_ {\tilde{\mathcal O}},$$
where $\tilde v$
is the extension of v by zero outside  ${\mathcal O}$ and ${\tilde{\mathcal O}}$ respectively.
Denoting $v(t)=u_t(t)$ and $U(t)=(u(t),v(t))^T,$ for any $t\ge 0,$, we can rewrite system \eqref{wave} in the abstract form \eqref{abstract_form}, with ${\mathcal H}=H_0^1(\Omega)\times L^2(\Omega)$,
$$
\mathcal A=\begin{pmatrix}
0 & Id \\
\Delta & -a\chi_{\mathcal O}
\end{pmatrix}
$$
and $\mathcal B$ and $F$ defined as
$$
\mathcal B \begin{pmatrix} u \\ v \end{pmatrix} = \begin{pmatrix} 0 \\ -\chi_{\tilde{\mathcal{O}}} v \end{pmatrix}, \qquad F(U(t))= \begin{pmatrix} 0 \\  u(t)|u(t)|^\beta \end{pmatrix}, \quad \forall\ t \geq 0.
$$
We know that $\mathcal{A}$ generates an exponentially stable $C_0$-semigroup $\{S(t)\}_{t\geq 0}$ (see e.g. \cite{K}), namely there exist $\omega,M>0$ such that
$$
||S(t)||_{\mathcal{L}(\mathcal H)} \leq Me^{-\omega t}, \quad \forall \ t\geq 0.
$$
Moreover, we consider the following functional:
$$\psi(u)=\frac{1}{\beta+2}\int_{\Omega} |u(x)|^{\beta+2} dx,$$
for any $u\in H_0^1(\Omega)$. For each $\beta\in \left( 0, \frac{4}{d-2}\right ]$, $\psi$ is well-defined by Sobolev's embedding theorem. Furthermore, $\psi$ is G\^ateaux differentiable for any $u\in H_0^1(\Omega)$ with G\^ateaux derivative given by
$$
D\psi(u)(v)=\int_{\Omega} |u(x)|^\beta u(x)v(x)dx, \quad \forall \ v\in H_0^1(\Omega).
$$
As in \cite{ACS}, it is possible to show that, for $\beta \in \left( 0, \frac{2}{d-2}\right]$, the functional   $\psi$ satisfies the assumptions (H1), (H2) and (H3).

We define the following energy functional:

$$E(t):= \frac 12 \int_\Omega |u_t(x,t)|^2 dx+\frac 12 \int_{\Omega} |\nabla u(x,t)|^2 dx-\psi(u(x,t))+\frac 12 \int_{t-\overline\tau}^t\int_{\tilde{\mathcal O}} |k(s)|\cdot |u_t(x,s)|^2 dx ds,
$$

as well as the functional

$$
{\mathcal E}(t):=\max \left \{  \frac 14 \max_{s\in [-\overline\tau,0] }\int_{\tilde {\mathcal O}} \vert g(x,s)\vert^2 \,ds, \ a
\max_{s\in [0, t] }E(s)  \right \}.
$$
Then, Theorem \ref{generaleCV} can be applied to system \eqref{wave}, under the assumption corresponding to \eqref{assumption_delay}, 
obtaining the well-posedness and exponential stability results for small initial data.

\begin{Remark}\label{Plate}{\rm
As another example we could consider the damped plate equation with source term (see e.g. \cite{Mustafa, PP_DCDSS} for the model details). The analysis is analogous to the wave case above.
Then, under suitable assumptions, the well-posedness and exponential stability results, for small initial data, hold for that model.}
\end{Remark}

\bigskip

\noindent {\bf Acknowledgements.} I would like to thank GNAMPA and UNIVAQ  for the support.

\end{document}